\RequirePackage{fix-cm}
\documentclass[smallextended]{svjour3}
\smartqed

\usepackage{amsmath}
\usepackage{amsfonts}
\usepackage{amssymb}
\usepackage[dvipspdf]{xcolor}

\def\peter#1 {\fbox {\footnote {\ }}\ \footnotetext { From Peter: {\color{red}#1}}}

\def\gohar#1 {\fbox {\footnote {\ }}\ \footnotetext { From Gohar: {\color{blue}#1}}}

\def\qi#1 {\fbox {\footnote {\ }}\ \footnotetext { From Qi: {\color{orange}#1}}}

\newcommand{\gf}{{\mathbb F}}

\begin{document}

\title{On the size of Kakeya sets in finite vector spaces}

\author{Gohar Kyureghyan \and Peter M\"uller   \and Qi Wang}

\institute{
G. Kyureghyan \at Institute of Algebra and Geometry, Faculty of Mathematics, Otto-von-Guericke University Magdeburg, Universit\"atsplatz 2, 39106, Magdeburg, Germany\\
\email{gohar.kyureghyan@ovgu.de}\\
P. M\"uller \at 
Institute for Mathematics,
University of W\"urzburg, Campus Hubland Nord,
97074 W\"urzburg, Germany\\ \email{peter.mueller@mathematik.uni-wuerzburg.de} \\
Q. Wang \at Institute of Algebra and Geometry, Faculty of Mathematics, Otto-von-Guericke University Magdeburg, Universit\"atsplatz 2, 39106, Magdeburg, Germany\\
\email{qi.wang@ovgu.de}
}

\date{Received: date / Accepted: date}

\maketitle

\begin{abstract}
For a finite field $\gf_q$, a Kakeya set $K$ is a subset of $\gf_q^n$ that contains a line in every direction. This paper derives new  upper bounds on the minimum size of Kakeya sets when $q$ is even.

\keywords{Kakeya set \and finite vector space \and Gold power function}

\subclass{11T30 \and 11T06} 
\end{abstract}

\section{Introduction}

Let $\gf_q$ be a finite field with $q$ elements. A {\em Kakeya set} $K \subset \gf_q^n$ is a set containing a line in every direction. 
More formally, $K \subset \gf_q^n$ is a Kakeya set if and only if for every ${\bf x} \in \gf_q^n$, there exists ${\bf y} \in \gf_q^n$ such that $\{ {\bf y} + t {\bf x} : t \in \gf_q \} \subset K$. Wolff in~\cite{wolff99} asked whether a lower bound of the form $|K| \geq C_n \cdot q^n$ holds for all Kakeya sets $K$, where $C_n$ is a constant depending
only on $n$. Dvir~\cite{Dvir08} first gave such a lower bound with $|K| \geq (1/n!) q^n$. Later Dvir, Kopparty, Saraf and Sudan improved the lower bound to $|K| \geq (1/2^n) q^n$ in~\cite{DKSS09} (see also~\cite{SS08}). It was shown in~\cite{DKSS09} that for any $n \geq 1$ there exists a Kakeya set $K \subset \gf_q^n$ with
\begin{equation}\label{eqn-upperbound}
  |K| \leq 2^{-(n-1)} q^n + O(q^{n-1}).
\end{equation}
For more information on Kakeya sets, we refer to a recent survey~\cite{Dvir12}.
 
When $q$ is bounded and $n$ grows, bound (\ref{eqn-upperbound}) is weak, and some recent papers improved the $O$-term in it to give better upper bounds for this case. 
The best currently known bound was obtained by Kopparty, Lev, Saraf and Sudan in~\cite{KLSS11}, following the ideas from~\cite{SS08,DKSS09} (see also~\cite{MT04}):  

\begin{theorem}\cite[Theorem 6]{KLSS11}\label{thm-levbound}
Let $n \geq 1$ be an integer and $q$ a prime power. There exists a Kakeya set $K \subset \gf_q^n$ with
$$
|K| < 
\left\{ \begin{array}{ll}
  2\left( 1+ \frac{1}{q-1} \right) \left(\frac{q+1}{2}\right)^n & \textrm{if $q$ is odd,} \\
  \frac{3}{2} \left( 1 + \frac{1}{q-1} \right) \left(\frac{2q+1}{3}\right)^n & \textrm{if $q$ is an even power of $2$,} \\
  \frac{3}{2} \left( \frac{2(q+\sqrt{q}+1)}{3} \right)^n & \textrm{if $q$ is an odd power of $2$}.
\end{array}\right.
$$
\end{theorem}

Theorem~\ref{thm-levbound} was proved by constructing a Kakeya set $K \subset \gf_q^n$ from a suitable function $f: \gf_q \rightarrow \gf_q$ as follows: For a given $t \in \gf_q$, set $$I_f(t) := \{ f(x) + tx \,|\, x \in \gf_q\}.$$ 
Further, define 
$$
K := \{(x_1, \ldots, x_j, t, 0, \ldots, 0)\,|\, 0 \leq j \leq n-1, t \in \gf_q, x_1, \ldots, x_j \in I_f(t) \}.
$$
If $f$ is a non-linear function, then $K$ is  a Kakeya set~\cite{KLSS11} of size 
\begin{equation}\label{eqn-kssize}
  |K| = \sum_{j=0}^{n-1} \sum_{t \in \gf_q} |I_f(t)|^j = \sum_{t \in \gf_q} \frac{|I_f(t)|^n-1}{|I_f(t)| - 1} .
\end{equation}
Clearly, to construct a small Kakeya set, we need to find a function $f: \gf_q \rightarrow \gf_q$ for which the sets $I_f(t)$ are small. Theorem~\ref{thm-levbound} was obtained by taking
\begin{itemize}
  \item[-] $f(x) = x^2$ for $q$ odd, since then $|I_f(t)| \leq (q+1)/2$ holds for all $ t \in \gf_q$;
  \item[-] $f(x) = x^3$ for $q$ an even power of $2$, since then $|I_f(t)| \leq (2q + 1)/3$ holds for all $t \in \gf_q$;
  \item[-] $f(x) = x^{q-2} + x^2$ for $q$ an odd power of $2$, since then $|I_f(t)| \leq 2(q + \sqrt{q} + 1)/3$ holds for all $t \in \gf_q$.
\end{itemize}

In~\cite{KLSS11}, it was also mentioned that it might be possible to choose better non-linear functions $f: \gf_q \rightarrow \gf_q$ to improve the bounds in Theorem~\ref{thm-levbound}. 

In this paper, we investigate this idea further and derive indeed  better upper bounds on the size of Kakeya sets $K \subset \gf_q^n$, when $q$ is even. 
 Our main result is
$$
|K| < 
\left\{ \begin{array}{ll} \vspace*{0.3cm}
\frac{2q}{q + \sqrt{q} - 2} \left( \frac{q + \sqrt{q}}{2} \right)^n  & \textrm{if $q$ is an even power of $2$,} \\
  \frac{8q}{5q +2\sqrt{q}-3} \left( \frac{5q+ 2\sqrt{q}+5}{8} \right)^n & \textrm{if $q$ is an odd power of $2$}.
\end{array}\right.
$$

In this paper we use the following result by Bluher~\cite{Blu04}:

\begin{theorem}\cite[Theorem 5.6]{Blu04}\label{thm-blu}
  Let $q = 2^m$ and $0 \leq i < m$ with $d = \gcd(i,m)$. Let $N_0$ denote the number of $b \in \gf_q^*$ such that $x^{2^i+1} + bx + b$ has no root in $\gf_q$.
\begin{itemize} \vspace*{0.1cm}
  \item[(i)] If $m/d$ is even, then $N_0 = \displaystyle{\frac{2^d(q-1)}{2(2^d+1)}}$. \vspace*{0.2cm}
  \item[(ii)] If $m/d$ is odd, then $N_0 = \displaystyle{\frac{2^d(q+1)}{2(2^d + 1)}}$.
\end{itemize}
\end{theorem}

\section{On Kakeya sets constructed using Gold power functions}

In this section, we use the Gold power functions $f(x)= x^{2^i+1}$ to derive  upper bounds on the minimum size of Kakeya sets $K \subset \gf_q^n$ with $q$ even.

Theorem \ref{thm-blu} allows us to determine explicitly the size of the image set $I_f(t) := \{f(x) + tx: x \in \gf_q \}$ with $f(x) = x^{2^i + 1}$ and $t \in \gf_q$.

\begin{proposition}\label{pro-issize}
  Let $q = 2^m$, $f(x) = x^{2^i + 1} \in \gf_q[x]$ with $0 \leq i < m$, and $d = \gcd(i,m)$. Set $I_f(t) := \{f(x) + tx: x \in \gf_q \}$ for $t \in \gf_q$. We have: 
  \begin{itemize}
    \item[(i)] if $m/d$ is even, then $|I_f(0)| = \displaystyle{1 + \frac{q - 1}{2^d+1}}$, and $|I_f(t)| = \displaystyle{ \frac{q+1}{2} + \frac{q-1}{2(2^d+1)}} $ for any $t \in \gf_q^*$;
    \item[(ii)] if $m/d$ is odd, then $|I_f(0)| = \displaystyle{q}$, and $|I_f(t)| = \displaystyle{ \frac{q-1}{2} + \frac{q+1}{2(2^d + 1)}} $ for any $ t \in \gf_q^*$.
  \end{itemize}
\end{proposition}

\begin{proof}
  For $t = 0$, we have 
  $$
  |I_f(0)| = 1 +  \displaystyle{\frac{2^m - 1}{\gcd(2^m-1,2^i+1)}} .
  $$ 
  From the well-known fact (e.g.~\cite[Lemma 11.1]{McEliece87}) that
  \begin{equation*}
    \gcd(2^m-1, 2^i + 1) = \left\{ \begin{array}{ll}
      1 & \textrm{ if $m/d$ is odd,} \\
      2^d + 1 & \textrm{ if $m/d$ is even,}
    \end{array}\right.
  \end{equation*}
the assertion on $|I_f(0)|$ follows.

For $t \in \gf_q^*$, by definition, we have
  \begin{eqnarray*}
    |I_f(t)| & = & | \{f(x) + tx: x \in \gf_q\}| \\
    & = & |\gf_q| - | \{c \in \gf_q^*: f(x) + tx + c \textrm{ has no root in $\gf_q$} \} | \\
    & = & q - N_0' .
  \end{eqnarray*}
  To make use of Theorem~\ref{thm-blu}, we transform $f(x) + tx + c$ following the steps in \cite{Blu04}. Since $t \ne 0$ and $c \ne 0$, let $x = \displaystyle{\frac{c}{t} z}$, then
  \begin{eqnarray*}
    \lefteqn{f(x) + tx + c} \\
    & = & x^{2^i + 1} + tx + c \\
    & = & \frac{c^{2^i +1}}{t^{2^i+1}} \left( z^{2^i+1} + \frac{t^{2^i+1}}{c^{2^i}} z + \frac{t^{2^i+1}}{c^{2^i}} \right).
  \end{eqnarray*}
 
  Since
  $$
  \left\{ \frac{t^{2^i + 1}}{c^{2^i}} : c \in \gf_q^* \right\} = \gf_q^*,
  $$
  we have $N_0' = N_0$, where $N_0$ denotes the number of $b \in \gf_q^*$ such that $x^{2^i+1} + bx + b$ has no root in $\gf_q$. The conclusion then follows from Theorem~\ref{thm-blu}.
\qed
 \end{proof}

Proposition \ref{pro-issize} shows that the smallest Kakeya sets constructed using Gold power functions are achieved
with $i =m/2$ for an even $m$, and $i=0$ for an odd $m$. 
The discussion below shows that the choice $i=m/2$ implies a better upper bound on Kakeya sets compared with the one given in Theorem  \ref{thm-levbound}. 
The idea to use $f(x) = x^{2^{m/2}+1}$ to improve the bound in Theorem
\ref{thm-levbound} appears in \cite{kyureg-wang-wcc}, and was independently suggested by David Speyer in \cite{mathoverflow}.
Observe that $f(x) = x^3$ chosen in~\cite{KLSS11} to prove the bound for $m$ even is the Gold power function with $i = 1$ and $d =1$.

When $m/d$ is odd, $|I_f(0)| = q$, and therefore the  bound obtained by the Gold power functions cannot be good for large $n$. However, for small values of $n$, it is better than the one of 
Theorem \ref{thm-levbound} \cite{kyureg-wang-wcc}.

Next  consider the function $f(x) = x^{2^{m/2}+1}$. In particular, we show that this function yields a better upper bound on the minimum size of Kakeya sets in $\gf_q^n$ when $q$ is an even power
of $2$. First we present a direct proof for the size of the sets $\{ x^{2^{m/2}+1} +tx \,:\, x \in \gf_q \}, t \in \gf_q$. 

\begin{theorem}\label{meven}
Let $m$ be an even integer.  Then
$$
|I(0)| := |\{ x^{2^{m/2}+1}  \,:\, x \in \gf_q \}| = 2^{m/2},
$$
 and 
$$
|I(t)| := |\{ x^{2^{m/2}+1} +tx \,:\, x \in \gf_q \}| = \frac{2^m + 2^{m/2}}{2} 
$$ 
for any $t \in \gf_q^*$.

\end{theorem}
\begin{proof}
The identity on $I(0)$ is clear, since the  image set of the function $x \mapsto x^{2^{m/2}+1}$ is $\mathbb{F}_{2^{m/2}}$.  Let $t \in \gf_q^*$. Note  that $|I(t)| =|I(1)|$. Indeed, there is $s \in \gf_q$, such that $t = s^{2^{m/2}}$ and then
$$
x^{2^{m/2}+1} +tx = s^{2^{m/2}+1}\cdot \left( (x/s)^{2^{m/2}+1} + (x/s)\right).
$$ 
Hence it is enough to compute $I(1)$. Let $Tr(x) = x^{2^{m/2}}+x$ be the trace map from $\gf_q$ onto its subfield  $\mathbb{F}_{2^{m/2}}$. Recall that $Tr$
is a $\mathbb{F}_{2^{m/2}}$-linear surjective map. 

Set $g(x) = x^{2^{m/2}+1} +x$. If $y,z \in \gf_q$ are such that
$$
g(z) = z^{2^{m/2}+1} +z = y^{2^{m/2}+1} +y = g(y),
$$
then $z = y + u$ for some $u \in \mathbb{F}_{2^{m/2}}$, since the image set of the function $x \mapsto x^{2^{m/2}+1}$ is $\mathbb{F}_{2^{m/2}}$.
Further, for any $u \in \mathbb{F}_{2^{m/2}}$
$$
g(y+u) = (y + u)^{2^{m/2}+1} +y+u = y^{2^{m/2}+1} +y + u(y^{2^{m/2}}+y) + u^2 + u.
$$
Hence, $g(y) = g(y+u)$ if and only if
$$
 u(y^{2^{m/2}}+y) + u^2 + u = u(Tr(y)+ u+1)=0.
$$
Consequently, two distinct elements $y$ and $z$ share the same image under the function $g$ if and only $Tr(y) \ne 1$ and $z = y + Tr(y)+1$. 
This shows that $g$ is injective on  the set $\mathcal{O}$ of elements from $\gf_q$ having trace 1, and 2-to-1 on  $\gf_q \setminus \mathcal{O}$,
completing the proof.

\qed \end{proof}

\begin{theorem}
Let $q=2^m$ with $m$ even and $n \geq 1$. There is a Kakeya set $K \subset \gf_q^n$ such that
$$
|K| < \frac{2q}{q + \sqrt{q} - 2} \left( \frac{q + \sqrt{q}}{2} \right)^n.
$$
\end{theorem}
\begin{proof}
The statement follows from (\ref{eqn-kssize}) and Theorem \ref{meven}.
\qed \end{proof}

\section{On Kakeya sets constructed using the function $x \mapsto x^4+x^3$}

In this section we obtain an upper bound on the minimum size of Kakeya sets constructed using
the function $x \mapsto  x^4+x^3$ on $\gf_q$. For every $t \in \gf_q$, let $g_t : \gf_q \to \gf_q$ be defined by
$$
g_t(x) := x^4 +x^3 + tx.
$$
Next we study the image sets of functions $g_t(x)$. Given $y \in \gf_q$, let $g_t^{-1}(y)$ be the set of preimages of $y$,
that is
$$
g_t^{-1}(y) := \{ x \in \gf_q \,|\, g_t(x) =y \}.
$$
Further, for any integer $k \geq 0$ put $\omega_t(k)$ to denote the number of elements in $\gf_q$ having exactly $k$ preimages
under $g_t(x)$, that is
$$
\omega_t(k) := |\{ y \in \gf_q \,:\, |g_t^{-1}(y)| = k \}|.
$$
Note that  $\omega_t(k) =0$ for all $k \geq 5$, since the degree of $g_t(x)$ is 4. The next lemma establishes the value of
$\omega_t(1)$:

\begin{lemma}\label{omega-1}
Let $q = 2^m$ and $t \in \gf_q^*$. Then
\begin{itemize}
\item if $m$ is odd
$$
\omega_t(1) = \left\{ \begin{array}{ll} \vspace*{0.2cm}
  \frac{q+1}{3} & \mbox{ if } Tr(t) =0  \\
  \frac{q+4}{3}  & \mbox{ if } Tr(t) = 1
\end{array} \right .
$$
\item if $m$ is even
$$
\omega_t(1) = \left\{ \begin{array}{ll} \vspace*{0.2cm}
  \frac{q-1}{3} & \mbox{ if } Tr(t)=0 \\
  \frac{q+2}{3}  & \mbox{ if } Tr(t) =1.
\end{array} \right .
$$
\end{itemize}

\end{lemma}

\begin{proof}
Let $y \in \gf_q$ and $y \ne t^2$.
Then $t^{2^{m-1}}$ is not a solution of the following equation 
$$
h_{t,y}(x):= g_t(x) + y = x^4 +x^3 +tx +y =0.
$$
Observe that the number of the  solutions for the above equation 
is equal to the one of
$$
(t^2+y)x^4 +t^{2^{m-1}}x^2 + x +1 = x^4\cdot  h_{t,y}\left(\frac{1}{x} + t^{2^{m-1}} \right) =0.
$$
Hence  either $\omega_t(1)$ or $\omega_t(1)-1$ is equal to the number of elements $y \in \gf_q$ such that the affine polynomial
\begin{equation}\label{affine}
(t^2+y)x^4 +t^{2^{m-1}}x^2 + x +1
\end{equation}
has exactly one zero in $\gf_q$, depending on the number of preimages of $g_t(x)$ for $t^2$. Equation (\ref{affine}) has exactly 1 solution if and only if the linearized
polynomial 
$$
(t^2+y)x^4 +t^{2^{m-1}}x^2 + x
$$
has no non-trivial zeros, or equivalently
\begin{equation}\label{equat-bl}
u(x) := (t^2+y)x^3 +t^{2^{m-1}}x + 1
\end{equation}
has no zeroes. 

Since  $t^2+y \ne 0$, the number of zeroes of $u(x)$ is equal to the one of 
$$
\frac{1}{t^2+y} \cdot u\left(\frac{1}{t^{2^{m-1}}} z\right) = \frac{1}{t^{2^{m-1} + 1}} \left( z^3 + \frac{t^{2^{m-1} + 1}}{t^2 + y} z + \frac{t^{2^{m-1} + 1}}{t^2 + y} \right).
$$
Note that 
$$
\left\{ \frac{t^{2^{m-1} + 1}}{t^2 + y} : y \in \gf_q, y \ne t^2 \right\} = \gf_q^*.
$$
Hence by Theorem \ref{thm-blu} with $i=1$,  the number of elements $y \in \gf_q, y \ne t^2$, such that
(\ref{equat-bl}) has no zeros is
$$
\left\{ \begin{array}{ll} \vspace*{0.2cm}
  \frac{q+1}{3} & \mbox{ if $m$ is odd} \\ 
   \frac{q-1}{3} & \mbox{ if $m$ is even.}
\end{array} \right .
$$
To complete the proof, it remains to consider $y = t^2$. In this case
$$
g_t(x) + y = x^4 +x^3 +tx +t^2 = (x^2+t)(x^2+x+t),
$$
and therefore $g_t(x)+t^2$ has exactly one solution if $Tr(t) =1$ and exactly 3 solutions if $Tr(t)=0$.

\qed \end{proof}

\begin{lemma}\label{omega-3}
Let $q = 2^m$ and $t \in \gf_q^*$. Then
$$
\omega_t(3) = \left\{ \begin{array}{ll} \vspace*{0.2cm}
  1  & \mbox{ if } Tr(t) =0  \\
  0  & \mbox{ if } Tr(t) = 1.
\end{array} \right .
$$
\end{lemma}
\begin{proof}
The proof of Lemma \ref{omega-1} shows that for any $y \ne t^2$, the number of solutions for $h_{t,y}(x) =0$ is a power of 2.
Hence only $t^2$ may have 3 preimages under $g_t(x)$, which is the case if and only if $Tr(t)=0$.
\qed \end{proof}

The next lemma describes the behavior of the function $x^4+x^3$:
\begin{lemma}\label{omega0}
Let $q = 2^m$ and $k\geq 1$ an integer. Then
\begin{itemize}
\item if $m$ is odd
$$
\omega_0(k) = \left\{ \begin{array}{ll} \vspace*{0.2cm}
  q/2 & \mbox{ if } k=2  \\
   0  & \mbox{  otherwise,}
\end{array} \right .
$$
in particular, the cardinality of $I(0) := \{ x^4+x^3 : x \in \gf_q \}$ is $q/2$.
\item if $m$ is even
$$
\omega_0(k) = \left\{ \begin{array}{ll} \vspace*{0.2cm}
  1 & \mbox{ if } k=2  \\ \vspace*{0.2cm}
  \frac{2(q-1)}{3} & \mbox{ if } k=1 \\ \vspace*{0.2cm}
   \frac{(q-4)}{12} & \mbox{ if } k=4 \\ \vspace*{0.2cm}
  0  & \mbox{ otherwise.}
\end{array} \right .
$$
\end{itemize}
\end{lemma}
\begin{proof}
Note that $x^4+x^3 =0$ has 2 solutions.
Let $y \in \gf_q^*$. Then the steps of the proof for Lemma \ref{omega-1} show that the number of solutions of 
$$
x^4+x^3 +y =0
$$
is equal to the one of the affine polynomial
$$
a_y(x):= yx^4  + x +1.
$$
If the set of zeros of $a_y(x)$ is not empty, then the number of zeros of $a_y(x)$ is equal to the one of
the linearized polynomial 
$$l_y(x)  := yx^4 +x.
$$
If $m$ is odd, then $l_y(x)$ has exactly 2 zeroes for every $y \ne 0$, implying the statement for $m$ odd.
If $m$ is even, then $l_y(x)$ has only the trivial zero if $y$ is a non-cube in $\gf_q$, and otherwise
it has 4 zeroes. To complete the proof, it remains to recall that the number of non-cubes in $\gf_q$ is
$2(q-1)/3$.
\qed \end{proof}

Lemmas \ref{omega-1}--\ref{omega0} yield the following upper bound for the size of the image sets of the considered
functions:

\begin{theorem}\label{thm-new}
Let $q = 2^m$ with $m$ odd. For $t \in \gf_q$ set $I(t) := \{ x^4+x^3
+ tx : x \in \gf_q \}$. Let $v$ be the number of pairs $x,z\in\gf_q$
with $x^2+zx=z^3+z^2+t$. Then for $t\ne0$
\[
  \lvert I(t)\rvert=\frac{5}{8}q+\frac{q+1-v}{8}+\frac{\delta}{2}<
\frac{5}{8}q+\frac{2\sqrt{q}+5}{8},
\]
where $\delta=0$ or $1$ if $Tr(t)=0$ or $1$, respectively.
\end{theorem}
\begin{proof}
Note that
\[
\lvert I(t)\rvert = \omega_t(1) + \omega_t(2) + \omega_t(3) + \omega_t(4)
\]
and
\[
q = \omega_t(1) + 2 \cdot \omega_t(2) + 3\cdot \omega_t(3) + 4\cdot \omega_t(4).
\]
Let $v'$ be the number of distinct elements $x,y\in \gf_q$ with
$x^4+x^3+tx=y^4+y^3+ty$.
Clearly
\[
v' = 2\omega_t(2) + 6\omega_t(3) + 12\omega_t(4),
\]
hence
\[
\lvert I(t)\rvert=\frac{5q - v' + 3\omega_t(1) - \omega_t(3)}{8}.
\]
Setting $y=x+z$, we see that $x^4+x^3+tx=y^4+y^3+ty$ for $x\ne y$ is
equivalent to $x^2+zx=z^3+z^2+t$ for $z\ne0$. However, for $z=0$ this
latter equation has a unique solution, so $v=v'+1$.

Together with Lemmas \ref{omega-1}--\ref{omega0} we see that the size
of $I(t)$ is as claimed. The inequality follows from the Hasse bound for
points on elliptic curves, which in our case says that $\lvert
v-q\rvert\le 2\sqrt{q}$. (Note that the projective completion of the
curve $X^2+ZX=X^3+X^2+t$ has a unique point at infinity.)
\qed \end{proof}

The bound obtained in Theorem \ref{thm-new} can be stated also as follows
\begin{equation}\label{int-part}
 \lvert I(t)\rvert \leq \left\lfloor \frac{5}{8}q+\frac{2\sqrt{q}+5}{8} \right\rfloor,
\end{equation}
since $\lvert I(t)\rvert$ is an integer. Our numerical calculations show that for odd $1 \leq m \leq 13$ bound (\ref{int-part}) is sharp, 
that is for these $m$  there are elements $t \in \mathbb{F}_{2^m}$ for which equality holds in (\ref{int-part}).

\begin{theorem}
Let $q=2^m$ with $m$ odd and $n \geq 1$. There is a Kakeya set $K \subset \gf_q^n$ such that
$$
|K| < \frac{8q}{5q +2\sqrt{q}-3} \left( \frac{5q+ 2\sqrt{q}+5}{8} \right)^n.
$$
\end{theorem}

\begin{proof}
The statement follows from (\ref{eqn-kssize}) and Theorem \ref{thm-new}.
\qed \end{proof}



\end{document}